\documentclass{amsart}
\usepackage{amsmath,amssymb,amsthm,mathabx,verbatim,nicefrac}
\usepackage[all]{xy}
\usepackage{color}
\usepackage{graphicx}
\usepackage{cancel}
\usepackage{datetime}

\newtheorem{thm}{Theorem}[section]

\newtheorem{prop}[thm]{Proposition}

\newtheorem{lem}[thm]{Lemma}

\newtheorem{cor}[thm]{Corollary}
\newtheoremstyle{named}{}{}{\itshape}{}{\bfseries}{.}{.5em}{#3}
\theoremstyle{named}
\newtheorem*{namedtheorem}{Theorem}
\theoremstyle{remark}

\begin{document}

\title[On the automorphisms of $X_{ns}(p)$]{On the automorphisms of the non-split Cartan modular curves of prime level}
\author{Valerio Dose}
\email{dose@mat.uniroma2.it}
\address{Dipartimento di Matematica\\Universit\`a di Roma ``Tor Vergata''\\
Via della Ricerca Scientifica 1\\00133 Roma\\ITALY}

\subjclass[2010]{14G35,11G05}
\keywords{modular curves, ellipic curves, galois representations}

\begin{abstract}
We study the automorphisms of the non-split Cartan modular curves $X_{ns}(p)$ of prime level $p$. We prove that if $p\geq 37$ all the automorphisms preserve the cusps. Furthermore, if $p\equiv1\textnormal{ mod }12$ and $p\neq 13$, the automorphism group is generated by the modular involution given by the normalizer of a non-split Cartan subgroup of $\text{GL}_2(\mathbb F_p)$. We also prove that for every $p\geq 37$ such that $X_{ns}(p)$ has a CM rational point, the existence of an exceptional rational automorphism would give rise to an exceptional rational point on the modular curve $X_{ns}^+(p)$ associated to the normalizer of a non-split Cartan subgroup of $\text{GL}_2(\mathbb F_p)$.
\end{abstract}

\maketitle

\section*{Introduction}

Modular curves can be constructed as compactifications of certain quotients of the upper half complex plane. When the genus exceeds 1, it is a classical question to ask if all of their automorphisms are induced by an automorphism of the upper half complex plane. Here we make some progress towards answering this question for non-split Cartan modular curves.

Let $N$ be a positive integer and let $H$ be a subgroup of $\text{GL}_2(\mathbb Z/N\mathbb Z)$ with the property that the determinant homomorphism from $H$ to $\mathbb Z/N\mathbb Z^*$ is surjective. We write $X_H$ for the modular curve over $\mathbb Q$ associated to the subgroup $H$. Let $\Gamma_H$ be the subgroup of $\text{SL}_2(\mathbb Z)$ made up of the elements which reduce modulo $N$ to an element in $H$. The complex points of $X_H$ can be identified with the orbit space $\mathcal H\cup \mathbb Q\cup \{\infty\}/\Gamma_H$, where $\mathcal H$ is the upper half complex plane on which any element of $\text{SL}_2(\mathbb R)$ acts as a M\"obius transformation $\displaystyle\begin{pmatrix}a&b\\c&d\end{pmatrix}\tau=\frac{a\tau+b}{c\tau+d}$, for every $\tau\in \mathcal H$. Let $\text{Norm}(\Gamma_H)$ be the normalizer of $\Gamma_H$ in $\text{SL}_2(\mathbb R)$. An element of $\text{Norm}(\Gamma_H)$ induces naturally an automorphism of the curve $X_H$, hence the group $\displaystyle B(X_H)\mathop{=}^\text{def}\text{Norm}(\Gamma_H)/\Gamma_H$ can be considered as a subgroup of the automorphism group $\text{Aut}(X_H)$ of $X_H$. We call an element of $B(X_H)$ a \textit{modular automorphism} of $X_H$. A non-modular automorphism is also called \textit{exceptional}.

If $H$ is a Borel subgroup of $\text{GL}_2(\mathbb Z/N\mathbb Z)$, the curve $X_H$ is the classical modular curve $X_0(N)$ and $\Gamma_H$ is the classical congruence subgroup $\Gamma_0(N)$. For every positive integer $N$, the automorphism group of $X_0(N)$ has been determined in \cite{OggDe}\cite{KM}\cite{Elkies}\cite{Harrison108}, and if the genus of $X_0(N)$ is at least 2 and $N\neq 37,63,108$, we have $\text{Aut}(X_0(N))=B(X_0(N))$. If we take $N=p$ a prime number, the group $B(X_0(p))$ is generated by the Atkin-Lehner involution $w_p$. The automorphism group of $\displaystyle X_0^*(p)\mathop{=}^\text{def}X_0(p)/\langle w_p\rangle$ has been determined in \cite{BakerHas} and it is trivial when the genus exceeds $2$. When $X_0^*(p)$ has genus $2$, the curve $X_0^*(p)$ is hyperelliptic, so it must admit an involution.

In this paper, we concentrate on the case of $H$ being a non-split Cartan subgroup of $\text{GL}_2(\mathbb F_p)$, and we write $X_{ns}(p)$ for the modular curve associated to $H$. This curve always has a modular involution $w$, because of the fact that a non-split Cartan subgroup of $\text{GL}_2(\mathbb F_p)$ has index 2 in its normalizer. We have $B(X_{ns}(p))=\langle w\rangle$ (see Section \ref{secPresCusp}). Let $\displaystyle X_{ns}^+(p)\mathop{=}^\text{def}X_{ns}(p)/\langle w\rangle$ be the modular curve associated to the normalizer of a non-split Cartan subgroup of $\text{GL}_2(\mathbb F_p)$. The curve $X_{ns}(p)$ has genus smaller than $2$ when $p\leq7$. The same happens for $X_{ns}^+(p)$ when $p\leq 11$. It is not true that $\text{Aut}(X_{ns}(p))=\langle w\rangle$ whenever the genus is at least $2$. For example, the automorphism group of $X_{ns}(11)$, determined in \cite{Cat}, is the Klein four group generated by $w$ and an exceptional involution.

Here, extending techniques of \cite{OggHyp}, \cite{KM} and \cite{BakerHas}, we prove the following

\begin{namedtheorem}[Theorem \ref{hyp}]
If $p\geq 11$ the modular curve $X_{ns}(p)$ is not hyperelliptic. If $p\geq 13$ the modular curve $X_{ns}^+(p)$ is not hyperelliptic.
\end{namedtheorem}
\begin{namedtheorem}[Theorem \ref{aut}]
If $p\geq 37$, all the automorphisms of $X_{ns}(p)$ preserve the cusps. If $p\equiv 1\textnormal{ mod }12$ and $p\neq 13$, the automorphism group of $X_{ns}(p)$ is generated by the modular involution $w$.
\end{namedtheorem}

The interest in non-split Cartan modular curves comes mainly from Serre's uniformity conjecture, which is an important statement regarding Galois representations attached to elliptic curves. It is equivalent to assert that for almost all $p$ and for every maximal subgroup $H$ of $\text{GL}_2(\mathbb F_p)$, the modular curve $X_H$ has no rational points except the cusps and the points associated to elliptic curves with complex multiplication, usually called CM points. The only case left to prove, concerns modular curves associated to the normalizer of a non-split Cartan subgroup of $\text{GL}_2(\mathbb F_p)$.

We prove the following statement relating the existence of exceptional automorphisms of $X_{ns}(p)$ with the existence of non-CM rational points on $X_{ns}^+(p)$.

\begin{namedtheorem}[Theorem \ref{unifAut}]
Let $p\geq 37$ and assume $X_{ns}(p)$ has a CM rational point (i.e. $p$ is inert in at least one of the imaginary quadratic field of class number one). If there exists an exceptional automorphism of $X_{ns}(p)$ defined over $\mathbb Q$, then $X_{ns}^+(p)$ has a non-CM rational point.
\end{namedtheorem}

Therefore, Serre's uniformity conjecture implies the non-existence of rational exceptional automorphisms of $X_{ns}(p)$ for almost all $p$ that are inert in at least one of the imaginary quadratic field of class number one.
\bigskip

\noindent\textbf{Acknowledgments}. {I would like to thank Prof. Ren\'e Schoof for many useful remarks on earlier versions of this paper.}

\section{Hyperelliptic modular curves of type $X_{ns}(p)$ and $X_{ns}^+(p)$}

Let $p$ be a prime number. We recall that the modular curve $X_{ns}(p)$ has genus $0$ for $p\le5$ and is elliptic only for $p=7$. The modular curve $X_{ns}^+(p)$ has genus 0 for $p\le 7$ and is elliptic only for $p=11$ (\cite{BaranClass}).

\begin{thm}\label{hyp}
If $p\geq 11$ the modular curve $X_{ns}(p)$ is not hyperelliptic. If $p\geq 13$ the modular curve $X_{ns}^+(p)$ is not hyperelliptic.
\end{thm}
\begin{proof}
(Following \cite[p. 455, Theorem 3]{OggHyp}). Let $q\neq p$ be a prime number, let $\mathbb F_q$ be the finite field with $q$ elements and let $\overline{\mathbb F_q}$ be an algebraic closure. The points of $X_{ns}(p)(\overline{\mathbb F_q})$ parametrize pairs $(E,\varphi)$, where $E$ is an elliptic curve over $\overline{\mathbb F_q}$ and $\varphi$ is an isomorphism from the $p$-torsion points of $E$ to $\mathbb Z/p\mathbb Z\times \mathbb Z/p \mathbb Z$. Moreover, two pairs $(E,\varphi)$, $(E',\varphi')$ are parametrized by the same point of $X_{ns}(p)(\overline{\mathbb F_q})$ if and only if there exist an isomorphism $f$ from $E$ to $E'$ such that, on the $p$-torsion of $E$, we have $\varphi'\circ f=M\circ\varphi$, where $M$ is an element of a non-split Cartan subgroup $C$ of $\text{GL}_2(\mathbb F_p)$. The Frobenius automorphism $\sigma$ generating $\text{Gal}(\overline{\mathbb{F}_q}/\mathbb F_{q^2})$ acts on $X_{ns}(p)(\overline{\mathbb F_q})$ as $\sigma(E,\varphi)=(E^\sigma,\varphi\circ\sigma^{-1})\text{,}$
so that every elliptic curve $E$ over $\mathbb F_{q^2}$ gives a point $(E,\varphi)$ in $X_{ns}(p)(\mathbb F_{q^2})$ if and only if $\varphi\circ\sigma^{-1}\circ\varphi^{-1}$ is contained in $C$.

Take a supersingular elliptic curve $E$ over $\overline{\mathbb F_q}$. This means that the $q$-torsion of $E$ is trivial, that the endomorphism algebra $\text{End}(E)$ is an order in a quaternion algebra, and implies that $E$ can be defined over $\mathbb F_{q^2}$ (\cite[p. 144, Theorem 3.1, (a)]{Silverman}). The fact that multiplication by $q$ is purely inseparable implies that the $q^2$-th power Frobenius endomorphism $\varphi_{q^2}$ and multiplication by $q$ differ by an automorphism of $E$, which means that $\varphi_{q^2}=\zeta q$ in $\text{End}(E)$ and $\zeta$ is a root of unity. If $j\neq 0,1728$, we have $\varphi_{q^2}=\pm q$. If $j=0,1728$, the elliptic curve $E$ is isomorphic over $\overline{\mathbb F_q}$ to a curve $E'$ defined over $\mathbb F_q$ with equation $y^2+y=x^3$ or $y^2=x^3-x$. We have $\#E'(\mathbb F_q)=q+1$ (\cite[p. 154, Excercise 5.10, (b)]{Silverman} when $q\ge 5$, a direct count when $q=2,3$). Hence the $q$-th power Frobenius endomorphism of $E'$ has characteristic polynomial $x^2+q=0$ which implies that $\varphi_{q^2}=-q$ in $\text{End}(E)$.

This shows that every supersingular elliptic curve over $\overline{\mathbb F_q}$ is isomorphic to a supersingular elliptic curve $E$ defined over $\mathbb F_{q^2}$, with the property that $\varphi_{q^2}$ acts on the $p$-torsion points as multiplication by $q$ or $-q$. Which says that $\varphi_{q^2}$ acts on the $p$-torsion of $E$ as a scalar matrix of $\text{GL}_2(\mathbb F_p)$ and therefore as an element contained in every conjugate of the non-split Cartan subgroup $C$ in $\text{GL}_2(\mathbb F_p)$. Hence, $E$ gives a point in $X_{ns}(p)(\mathbb F_{q^2})$ for each conjugate of $C$ in $\text{GL}_2(\mathbb F_p)$, up to automorphisms of $E$, in the sense that a pair $(E,\varphi)$ gives the same point on $X_{ns}(p)$ as $(E,\varphi\circ f)$ for any automorphism $f$ of $E$. Among the automorphisms of any supersingular elliptic curve, we have multiplication by $-1$, which acts on torsion points as an element contained in $C$. Thus, since the index of a non-split Cartan subgroup in $\text{GL}_2(\mathbb F_p)$ is equal to $p(p-1)$ (see \cite{BaranClass}), we have the following inequality
$$\#X_{ns}(p)(\mathbb F_{q^2})\ge p(p-1)\cdot 2\cdot\sum_{
\small\begin{matrix} E\text{ over }\overline{\mathbb F_q}\\ \text{supersingular}\end{matrix}
}\frac1{\#\textnormal{Aut}(E)}
=\frac{p(p-1)(q-1)}{12}\text{,}$$
where in the last equality we used the Deuring-Eichler formula (\cite[p. 154, Excercise 5.9]{Silverman})
$$\sum_{
\small\begin{matrix} E\text{ over }\overline{\mathbb F_q}\\ \text{supersingular}\end{matrix}
} \frac1{\#\textnormal{Aut}(E)}=\frac{q-1}{24}\text{.}$$
Note that this lower bound on the number of supersingular points can also be easily deduced from \cite[Lemma 3.20, Lemma 3.21]{BJG}.

Suppose now that $X_{ns}(p)$ is hyperelliptic. Then we have
$$\frac{p(p-1)(q-1)}{12}\leq\#X_{ns}(p)(\mathbb F_{q^2})\leq 2(q^2+1)\text{,}$$
where the second inequality holds because $X_{ns}(p)$ is hyperelliptic and because the projective line over $\mathbb F_{q^2}$ has exactly $q^2+1$ points. Now, putting $q=2$ we have
$$ \frac{p(p-1)}{12}\leq\#X_{ns}(p)(\mathbb F_{4})\leq 10$$
and we obtain $p\leq 11$ which ends the proof for $X_{ns}(p)$ in the case $p\geq 13$. With an analogous reasoning, taking into account that the index of the normalizer of a non-split Cartan subgroup in $\text{GL}_2(\mathbb F_p)$ is equal to $\frac{p(p-1)}2$, we obtain the statement for $X_{ns}^+(p)$ in the case $p\geq 17$.

The case $p=11$ for $X_{ns}(p)$ and the case $p=13$ for $X_{ns}^+(p)$ can be worked out by explicitly counting the points over $\mathbb F_4$ of the respective curves. Recall that the Jacobian of $X_{ns}(p)$ is isogenous to the new part of the Jacobian of $X_0(p^2)$ and the Jacobian of $X_{ns}^+(p)$ is isogenous to the new part of the Jacobian of $\displaystyle X_0^*(p^2)\mathop{=}^\text{def}X_0(p^2)/\langle w_{p^2}\rangle$, where $w_{p^2}$ is the Atkin-Lehner involution (\cite{Chen}\cite{Edix}). Then we can count the points over $\mathbb F_4$ by looking at the roots of the characteristic polynomial of a Frobenius at $2$ on the Jacobian of $X_0(p^2)$, and using the well known formula
$$
\#C(\mathbb F_{q^r})=q^r+1-\sum_{i=1}^{2g}\alpha_i^r
$$
where $C$ is a non-singular projective curve of genus $g$ over the finite field $\mathbb F_q$ with $q$ prime, and the $\alpha_i$'s are the roots of the characteristic polynomial of the Frobenius automorphism at $q$, acting on the Jacobian variety of $C$.

Looking at the tables in \cite{SteinTab} of weight-2 newforms for $\Gamma_0(121)$, we see that the Jacobian of $X_{ns}(11)$ is isogenous to the product of four elliptic curves for which the traces of a Frobenius automorphism at 2 are respectively $-1$,$0$,$1$,$2$, so that the roots of the characteristic polynomial of a Frobenius at 2 acting on the Jacobian of $X_{ns}(11)$ are $\frac{-1\pm\sqrt{-7}}2$,$\pm\sqrt{-2}$,$\frac{1\pm\sqrt{-7}}2$,$1\pm i$. Hence $\#X_{ns}(11)(\mathbb F_4)=15>10$.

Now, in the same tables, we look at the weight-2 newforms for $\Gamma_0(169)$. We see that the Jacobian of $X_{ns}^+(13)$ is isogenous to a simple abelian variety $A$ over $\mathbb Q$ of dimension $3$. The eigenvalues of the Hecke operator $T_2$ acting on $A$ are the roots $a_i$, with $i=1,2,3$, of the polynomial $x^3 + 2x^2 - x - 1$. Then the characteristic polynomial of a Frobenius at 2 is the product
$$\prod_{i=1}^3(x^2-a_ix+2)\text{.}$$
It allows us to compute $\#X_{ns}^+(13)(\mathbb F_4)=11>10$.
\end{proof}

\section{Endomorphisms of the jacobian variety of $X_{ns}(p)$}

In this section, for any abelian variety $A$ over a field $K$, an endomorphism of $A$ is an endomorphism defined over an algebraic closure $\overline K$ of $K$, and we write $\text{End}(A)$ for the ring of endomorphisms of $A$ defined over $\overline K$. The endomorphism algebra $\text{End}(A)\otimes \mathbb Q$ is an invariant of the isogeny class of $A$ over $\overline K$. The minimal field where every endomorphism of $A$ is defined is an invariant of the isogeny class over $K$ of $A$. We say that $A$ is of CM-type if $\text{End}(A)\otimes\mathbb Q$ contains a commutative semi-simple $\mathbb Q$-algebra of dimension equal to $2\cdot\text{dim }A$.

\begin{prop}\label{unr}
Let $A$ be a semi-stable abelian variety over a field $K$, complete with respect to a discrete valuation. Then every endomorphism of $A$ is defined over an unramified extension of $K$.
\end{prop}
\begin{proof}
\cite[p. 556, Theorem 1.1]{RibetEnd}.
\end{proof}

\begin{prop}\label{quadext}
Let $A$ be an abelian variety over a field $K$. Let $E$ be a subalgebra of $\text{End}_K(A)\otimes\mathbb{Q}$ such that $[E:\mathbb{Q}]=\text{dim }A$ and suppose that $E$ is a product $E_1\times\dots\times E_t$ of totally real number fields. This gives rise to a decomposition $A\sim A_1\times\dots\times A_t$ of abelian varieties up to isogeny over $K$. Suppose that no factor $A_i$ is of CM-type. Then $E$ is its own commutant in $\text{End}(A)\otimes\mathbb{Q}$ and every endomorphism of $A$ is defined over a compositum of quadratic extensions of $K$.
\end{prop}
\begin{proof}
\cite[p. 557, Theorem 2.3]{RibetEnd}
\end{proof}

Let $J_0(N)$ be the Jacobian variety of the modular curve $X_0(N)$. It is an abelian variety over $\mathbb Q$.

\begin{prop}\label{semist}
The abelian variety $J_0(N)$ is semi-stable at each prime $p$ such that $p^2$ does not divide $N$.
\end{prop}
\begin{proof}
\cite[p. 286, Theorem 6.9]{DelRap}.
\end{proof}

For every integer $N>0$, there exists  a set of representatives $f_1,\dots,f_t$ of $\text{Gal}(\overline{\mathbb Q}/\mathbb Q)$-conjugacy classes of normalized weight-2 cuspforms for $\Gamma_0(N)$, such that every $f_i$ is a newform at some level $M_i$ dividing $N$, and $\{f_i(n\tau),n|N/M_i,i=1\dots t\}$ with their $\text{Gal}(\overline{\mathbb Q}/\mathbb Q)$-conjugates form a basis of the complex vector space of weight-2 cuspforms for $\Gamma_0(N)$. There is a decomposition of abelian varieties up to isogeny over $\mathbb Q$
\begin{equation}\label{dec}
J_0(N)\sim A_1^{m_i}\times\dots\times A_t^{m_t}
\end{equation}
where $\displaystyle A_i\mathop{=}^\text{def}A_{f_i}$ is the abelian variety up to isogeny over $\mathbb Q$ associated to the $\text{Gal}(\overline{\mathbb Q}/\mathbb Q)$-conjugacy class of $f_i$, and $m_i$ is the number of divisors of $N/M_i$  (see \cite[Section 6.6]{DS}). Every $A_i$ is a simple abelian variety over $\mathbb Q$.

Let $J_0^C(N)$ be the abelian variety up to isogeny over $\mathbb Q$ which is a product of the factors $A_i^{m_i}$ such that $A_i$ is of CM-type and let $J_0^H(N)$ be the product of the remaining factors of the decomposition above. We recall here that if $A_i$ is of CM-type, then it is isogenous to a product of copies of an elliptic curve with complex multiplication by an imaginary quadratic field $K$ of discriminant $-D$ and $f_i$ is a weight-2 cuspform for $\Gamma_0(D\cdot \text{N}(\mathfrak c))$, where $\mathfrak c$ is the conductor of a primitive Hecke character of $K$ (see \cite[p. 138, Proposition 1.6]{ShimuraClassFields}\cite{ShimuraCM}). Any factor $A_i$ of $J_0^H(N)$ can't be isogenous to any factor $A_j$ of $J_0^C(N)$ because $A_i$ and $A_j$ have different endomorphism algebras. Hence, we have $J_0(N)\sim J_0^H(N)\times J_0^C(N)$ and $\text{End}(J_0(N))\otimes\mathbb Q\cong(\text{End}(J_0^H(N))\otimes \mathbb Q)\times(\text{End}(J_0^C(N))\otimes\mathbb Q)$.

Let $p$ be an odd prime number. We call $K(p)=\mathbb{Q}(\sqrt p)$ if $p\equiv1\text{ mod }4$ and $K(p)=\mathbb{Q}(\sqrt{-p})$ if $p\equiv3\text{ mod }4$.
\begin{prop}\label{fielddef0}
Every endomorphism of $J_0^H(p^2)$ is defined over $K(p)$.
\end{prop}
\begin{proof}
Proposition \ref{quadext} tells us that endomorphisms of $J_0^H(p^2)$ can be defined over a field $K$ which is a compositum of quadratic extensions of $\mathbb Q$. Furthermore, Propositions \ref{unr} and \ref{semist} imply that $K$ can be taken unramified outside $p$. Thus, we can take $K=\mathbb Q(\sqrt p)$ if $p\equiv 1\text{ mod }4$ and $K=\mathbb{Q}(\sqrt{-p})$ if $p\equiv3\text{ mod }4$.
\end{proof}

Let $g_0(N)$, $g_0^H(N)$, $g_0^C(N)$ be respectively the dimensions of $J_0(N)$, $J_0^H(N)$, $J_0^C(N)$.

Let $J_{ns}(p)$ be the Jacobian of the modular curve $X_{ns}(p)$. The abelian variety $J_{ns}(p)$ is isogenous over $\mathbb Q$ to the new part of $J_0(p^2)$ (see \cite{Chen}, \cite{Edix}). Let $J_0^\text{new}(p^2)$ and $J_0^\text{old}(p^2)$ be respectively the new and the old part of $J_0(p^2)$ and let's fix an isomorphism $\varphi$ from $\text{End}(J_{ns}(p))\otimes\mathbb Q$ to $\text{End}(J_0^\text{new}(p^2))\otimes\mathbb Q$. We say, with abuse of notation, that an automorphism $v$ of $X_{ns}(p)$, induces the invertible element $\displaystyle v\mathop{=}^\text{def}\varphi(v)$ of
\begin{align*}
\text{End}(J_0^\text{new}(p^2))\otimes\mathbb Q&\subset(\text{End}(J_0^\text{new}(p^2))\otimes\mathbb Q)\times(\text{End}(J_0^\text{old}(p^2))\otimes\mathbb Q)\\
&\subset\text{End}(J_0(p^2))\otimes\mathbb Q
\end{align*}
where the first inclusion is given by sending $v$ to the invertible element that acts like $v$ on the new part and as the identity on the old part.

We can then define in an analogous way $J_{ns}^H(p)$, $J_{ns}^C(p)$, as factors of the new part of the decomposition (\ref{dec}). We also define the dimensions $g_{ns}(p)$, $g_{ns}^H(p)$, $g_{ns}^C(p)$. We have $J_{ns}(p)\sim J_{ns}^H(p)\times J_{ns}^C(p)$ and $\text{End}(J_{ns}(p))\otimes\mathbb Q\cong(\text{End}(J_{ns}^H(p))\otimes \mathbb Q)\times(\text{End}(J_{ns}^C(p))\otimes\mathbb Q)$.

\begin{thm}\label{thm1}
If $g_{ns}(p)>p$ every automorphism of $X_{ns}(p)$ is defined over $K(p)$.
\end{thm}
\begin{proof}
(Following \cite[p. 55, Lemma 1.4]{KM}) Let $u$ be an automorphism of $X_{ns}(p)$ and $\sigma$ a element of $\text{Gal}(\overline{\mathbb Q}/K(p))$. We define $v=u^\sigma u^{-1}$ and let $d$ be the order of $v$ in the automorphism group of $X_{ns}(p)$. We call $Y$ the quotient of $X_{ns}(p)$ by the automorphism $v$. Let $g_Y$ be the genus of $Y$. We have that $v$ induces an automorphism of the Jacobian $J_{ns}(p)\sim J_{ns}^H(p)\times J_{ns}^C(p)$ which acts trivially on $J_{ns}^H(p)$ by Proposition \ref{fielddef0}, because $J_{ns}^H(p)$ is a factor of $J_0^H(p^2)$. Thus we have the following morphisms of abelian varieties with finite kernel
\begin{align*}
J_{ns}^H(p)&\rightarrow J_{ns}^H(p)\times J_{ns}^C(p)\rightarrow J_{ns}(p)\\
x&\mapsto (x,0)
\end{align*}
where the image of the composition is contained in $\displaystyle J_{ns}(p)^v\mathop{=}^\text{def}\{D\in J_{ns}(p)\text{ : }vD=D\}$. Hence there are morphisms of group varieties
$$J_{ns}^H(p)\longrightarrow J_{ns}(p)^v \longrightarrow \text{Jac}(Y)$$
where the second map is the push-forward of divisors, and the composition is a morphism of abelian varieties. Furthermore, both maps have finite kernel, hence $g_Y\geq g_{ns}^H(p)$ and we have by the Riemann-Hurwitz formula
$$g_{ns}(p)-1\geq d(g_Y-1)\geq 2g_{ns}^H(p)-2\text{,}$$
where in the second inequality we supposed $d>1$. Now using $g_{ns}(p)=g_{ns}^H(p)+g_{ns}^C(p)$ we obtain $g_{ns}(p)\leq 2g_{ns}^C(p)+1$.
Since $J_{ns}(p)$ is isogenous over $\mathbb{Q}$ to the new part of $J_0(p^2)$ we have $g_{ns}^C(p)=g_0^C(p^2)$. Furthermore $g_0^C(p^2)=0$ if $p\equiv1\text{ mod }4$ and $g_0^C(p^2)=h(-p)$ if $p\equiv3\text{ mod }4$, where $h(-p)$ is the class number of $\mathbb Q(\sqrt{-p})$. Indeed, in the decomposition (\ref{dec}) with $N=p^2$, if a factor is of CM-type, then it is the product of copies of an elliptic curve with complex multiplication by an imaginary quadratic field $K$ of discriminant $-p$. This leaves no possibilities if $p\equiv1\text{ mod }4$ and it implies $K=\mathbb Q(\sqrt{-p})$ when $p\equiv3\text{ mod }4$. In the latter case the dimension is $h(-p)$ because by the the theory of complex multiplication, the elliptic curves with complex multiplication by $\mathbb Q(\sqrt{-p})$ are all isogenous, $\text{Gal}(\overline{\mathbb Q}/\mathbb Q)$-conjugate and their cardinality is $h(-p)$.

Then we can estimate $h(-p)$ using Dirichlet's class number formula:
$$h(-p)=-\frac1p\sum_{m=1}^{p-1}m\left(\frac mp\right)\leq \frac1p\sum_{m=1}^{p-1}m=\frac{p-1}2$$
where the first equality is \cite[p. 51, (19)]{Davenport} (we are assuming $p\equiv 3\text{ mod }4$ and $p\neq 3$), and in the inequality we used that the Legendre symbol $\displaystyle\left(\frac mp\right)$ takes values in $\{\pm 1\}$. This implies $g_{ns}(p)\leq p$. Contradiction.
\end{proof}
\begin{cor}\label{Kp}
If $p\geq 11$ every automorphism of $X_{ns}(p)$ is defined over $K(p)$.
\end{cor}
\begin{proof}
By \cite[Theorem 7.2]{BaranClass}, we have that $g_{ns}(p)>p$ for $p\geq 19$. Regarding $p=13$, $17$, they are both congruent to $1$ modulo $4$, so in these cases $g_{ns}^C(p)=g_0^C(p^2)=0$ and we can use Proposition \ref{fielddef0} together with the isogeny over $\mathbb{Q}$ between $J_{ns}(p)$ and the new part of $J_0(p^2)$. The automorphisms of $X_{ns}(11)$ are explicitly computed in \cite{Cat} and they are all defined over $\mathbb Q$.
\end{proof}
The modular curve $X_{ns}(7)$ has genus $1$, and it has the automorphisms defined as translation by points. Thus, in this case there are automorphisms defined over larger number fields.

\begin{prop}\label{HeckeQ}
If an automorphism of $X_{ns}(p)$ is defined over $\mathbb Q$, then it induces an element of $\text{End}(J_0^H(p^2))\otimes\mathbb{Q}$ contained in the Hecke algebra.
\end{prop}
\begin{proof}
Since $J_{ns}(p)$ is isogenous over $\mathbb{Q}$ to the new part of $J_0(p^2)$, an automorphism $u$ of $X_{ns}(p)$ defined over $\mathbb Q$ induces an element of $\text{End}_{\mathbb Q}(J_0(p^2))\otimes\mathbb Q\cong(\text{End}_{\mathbb Q}(J_{ns}^H(p))\otimes \mathbb Q)\times(\text{End}_{\mathbb Q}(J_{ns}^C(p))\otimes\mathbb Q)$, as explained above. Since $J_0(p^2)$ has good reduction outside $p$, the reduction modulo $l$ map is injective on the endomorphism ring of $J_0(p^2)$ for every prime number $l\neq p$. Hence, by the Eichler-Shimura relations, $u$ commutes with every Hecke operator $T_l$ with $l\neq p$, because $u$ is defined over $\mathbb Q$. Moreover, the Hecke algebra acting on the new part is generated by the operators $T_l$ with $l\neq p$, thus $u$ commutes with the whole Hecke algebra. Then, to prove the Proposition, we note that every factor of $J_0^H(p^2)$ is without complex multiplication and so the Hecke algebra is its own commutant in $\text{End}(J_0^H(p^2))\otimes \mathbb Q$ (Proposition \ref{quadext}).
\end{proof}
\begin{lem}
If $g_{ns}(p)>p$ every automorphism of $X_{ns}(p)$ defined over $\mathbb Q$ commutes with the modular involution $w$.
\end{lem}
\begin{proof}
Let $u$ be an automorphism of $X_{ns}(p)$ defined over $\mathbb Q$. Then, by the previous Proposition, $u$ and $w$ induce elements in the Hecke algebra of $J_0^H(p^2)$, which is a commutative algebra. Hence, if we define $v=uwu^{-1}w^{-1}$ we have that $v$ induces an invertible element of $\text{End}(J_0(p^2))\otimes\mathbb Q$ which acts trivially on $J_0^H(p^2)$ and $J_{ns}^H(p)$. Then we can use the same proof as for Theorem $\ref{thm1}$ to see that $v$ is trivial if $g_{ns}(p)>p$.
\end{proof}
\begin{cor}\label{Qcomw}
If $p\geq 11$ every automorphism of $X_{ns}(p)$ defined over $\mathbb Q$ commutes with the modular involution $w$.
\end{cor}
\begin{proof}
As for Corollary \ref{Kp}, we have that $g_{ns}(p)>p$ for $p\geq 19$. Regarding $p=13$, $17$, in these cases $g_{ns}^C(p)=g_0^C(p^2)=0$. Moreover $g_{ns}(p)\geq 2$, so the natural map from the automorphism group of $X_{ns}(p)$ to the automorphism group of $J_{ns}(p)$ is injective (\cite[Lemma 2.1]{BakerHas}),  and we can use Proposition \ref{HeckeQ} together with the fact that the Hecke algebra is commutative. The automorphism group of $X_{ns}(11)$ is the Klein four group (see \cite{Cat}).
\end{proof}
\begin{cor}
If $p\ge 11$ every automorphism of $X_{ns}(p)$ defined over $\mathbb Q$ induces an automorphism of $X_{ns}^+(p)$ defined over $\mathbb Q$.
\end{cor}

\section{Automorphisms of $X_{ns}(p)$ preserving the cusps}\label{secPresCusp}

Let $X_H$ be a modular curve. Then the orbit space $\mathcal H\cup\mathbb Q\cup\{\infty\}/\Gamma_H$ can be identified with $X_H(\mathbb C)$. Let $E\subset \mathcal H$ be the set of ramification points of the natural quotient map $\pi:\mathcal H\rightarrow \mathcal H/\Gamma_H$. This is a discrete set in $\mathcal H$ and it is the counter image of the set of elliptic points in $X_H(\mathbb C)$. For every $z\in \mathcal H\setminus E$ there exists a neighborhood $U\ni z$ such that $\gamma_i(U)\cap \gamma_j(U)=\emptyset$ for every $\gamma_i\neq\pm\gamma_j$ in $\Gamma_H$. This means that $\mathcal H\setminus E$ is a proper covering space of $X_H(\mathbb C)$ without the cusps and the elliptic points.

\begin{prop}
Let $u$ be an automorphism of $X_H$. If $u$ preserves the set of cusps and preserves the set of elliptic points of $X_H$, then $u$ is induced by an automorphism of the Riemann surface $\mathcal H$ preserving $E$.
\end{prop}
\begin{proof}
Let $C$ be the set of cusps on $X_H$. Note that an automorphism $u$ of $X_H$ preserving the cusps and the elliptic points, defines naturally an automorphism of the open Riemann surface $X_H(\mathbb C)\setminus (C\cup \pi(E))$. If $X_H$ has no elliptic points ($E=\emptyset$), then the quotient $\mathcal H\rightarrow \mathcal H/\Gamma_H\cong X_H(\mathbb C)\setminus C$ is the universal cover of $X_H(\mathbb C)\setminus C$. Then $u$ automatically lifts to an automorphism of $\mathcal H$. If instead $E\neq \emptyset$, to lift $u$ to $\mathcal H\setminus E$, we need the push forward by $u\circ \pi$ of the fundamental group of $\mathcal H\setminus E$ to be contained in its push forward by $\pi$. But this is true because $u$ actually extends to an automorphism of $X_H(\mathbb C)$ preserving $C$ and $\pi(E)$. Hence $u$ lifts to an automorphism of $\mathcal H\setminus E$, which extends uniquely to $\mathcal{H}$.
\end{proof}

Let $p$ be a prime number.

\begin{prop}
Suppose that the genus of $X_{ns}(p)$ is at least 2. Let $\Gamma_{ns}(p)$ be the subgroup of $\textnormal{SL}_2(\mathbb Z)$ made up of the elements which reduce modulo $p$ to an element in a non-split Cartan subgroup of $\textnormal{GL}_2(\mathbb F_p)$. Then the subgroup $$B(X_{ns}(p))=\textnormal{Norm}(\Gamma_{ns}(p))/\Gamma_{ns}(p)\subset\textnormal{Aut}(X_{ns}(p))$$
is generated by the modular involution $w$.
\end{prop}
\begin{proof}
Since the genus of $X_{ns}(p)$ is at least 2, the group $B(X_{ns}(p))$ must be finite. This means that $\text{Norm}(\Gamma_{ns}(p))$ is commensurable with $\text{SL}_2(\mathbb Z)$ in the sense that $\text{Norm}(\Gamma_{ns}(p))\cap\text{SL}_2(\mathbb Z)$ has finite index in both $\text{Norm}(\Gamma_{ns}(p))$ and $\text{SL}_2(\mathbb Z)$. Hence $\text{Norm}(\Gamma_{ns}(p))$ is a group acting on lattices of $\mathbb R^2$ that are commensurable with $\mathbb Z\times \mathbb Z$, with the action given by right multiplication of row vectors in the lattice. Lattices commensurable with $\mathbb Z\times\mathbb Z$ are the ones generated by two vectors $(a,b)$, $(c,d)$ with $a,b,c,d\in\mathbb Q$, and if we consider them up to multiplication by a scalar (i.e. homothety), they all have a basis of the form $\displaystyle\{(M,\frac gh),(0,1)\}$ with $M\in\mathbb Q^+$, $g,h\in\mathbb Z$, $\text{mcd}(g,h)=1$ and $0\le g<h$ (see \cite{Conway} for more details).

Since $\displaystyle\Gamma_{ns}(p)\subset\Gamma(p)=\left\{\gamma\in\text{SL}_2(\mathbb Z)\text{ s. t. }\gamma\equiv\text{Id mod }p\right\}$, the lattices fixed by $\Gamma_{ns}(p)$ are a subset of the lattices fixed by $\Gamma(p)$. The lattices fixed by $\Gamma(p)$ are those who belong to one of the following type (see \cite[Lemma 4.1]{LungLang2002})
\begin{itemize}
\item $M=\frac 1p$, $h=p$, $g=0,\dots,p-1$, which gives a lattice homothetic to $\langle (1,g),(0,p)\rangle$;
\item $M=\frac 1p$, $h=1$, $g=0$, which gives a lattice homothetic to $\langle (1,0),(0,p)\rangle$;
\item $M=p$, $h=1$, $g=0$, which gives the lattice $\langle (p,0),(0,1)\rangle$;
\item $M=1$, $h=1$, $g=0$, which gives the lattice $\langle (1,0),(0,1)\rangle=\mathbb Z\times\mathbb Z$.
\end{itemize}
Recall that $\Gamma_{ns}(p)$ is the subgroup of $\text{SL}_2(\mathbb Z)$ made up of the elements $\gamma$ such that
$$\gamma\equiv\begin{pmatrix}x&\alpha y\\y&x\end{pmatrix}\text{ mod }p$$
for some $(x,y)\in\mathbb F_p^2$ and a non-square element $\alpha\in\mathbb F_p$. Then, a straightforward computation modulo $p$ shows that $\gamma$ fixes lattices of the first three types, only if $\gamma\in\Gamma(p)$, so that the only lattice up to homothety fixed by $\Gamma_{ns}(p)$ is $\mathbb Z\times\mathbb Z$.

The normalizer in $\text{SL}_2(\mathbb R)$ of $\Gamma_{ns}(p)$ must preserve, by definition, the set of lattices that are fixed by $\Gamma_{ns}(p)$. Thus, also $\text{Norm}(\Gamma_{ns}(p))$ fixes $\mathbb Z\times \mathbb Z$. The stabilizer of $\mathbb Z\times \mathbb Z$ is $\text{SL}_2(\mathbb Z)$, implying $\text{Norm}(\Gamma_{ns}(p))\subset \text{SL}_2(\mathbb Z)$. Therefore, $\text{Norm}(\Gamma_{ns}(p))$ is made up of the elements  of $\text{SL}_2(\mathbb Z)$ which reduce modulo $p$ to an element in the normalizer of a non-split Cartan subgroup of $\text{GL}_2(\mathbb F_p)$, which proves the Proposition.
\end{proof}

\begin{cor}\label{ellcuspModular}
Suppose that the genus of $X_{ns}(p)$ is at least 2. If $p\equiv 1\textnormal{ mod }12$, the only nontrivial automorphism of $X_{ns}(p)$ preserving the cusps is the modular involution $w$. If $p\notequiv 1\textnormal{ mod }12$, the only nontrivial automorphism of $X_{ns}(p)$ preserving the cusps and the elliptic points, is the modular involution $w$. The only automorphism of $X_{ns}^+(p)$ preserving the cusps and the elliptic points is the identity.
\end{cor}
\begin{proof}
Let $\Gamma_{ns}(p)$ be the subgroup of $\textnormal{SL}_2(\mathbb Z)$ made up of the elements which reduce modulo $p$ to an element in a non-split Cartan subgroup of $\textnormal{GL}_2(\mathbb F_p)$, so that $X_{ns}(p)(\mathbb C)\cong\mathcal H\cup\mathbb Q\cup\{\infty\}/\Gamma_{ns}(p)$. Thus, $X_{ns}(p)$ has elliptic points if and only if $\Gamma_{ns}(p)$ contains an element with characteristic polynomial equal to $x^2+1$ or $x^2+x+1$. Recall that the elements of a non-split Cartan subgroup of $\text{GL}_2(\mathbb F_p)$ have irreducible characteristic polynomials. If $p\equiv 1\textnormal{ mod }4$ and $p\equiv1\textnormal{ mod }3$, the two polynomials above have both roots over $\mathbb F_p$. Hence, when $p\equiv 1\textnormal{ mod }12$ there are no elliptic points on $X_{ns}(p)$. The Corollary then follows from the two previous Propositions.
\end{proof}

\section{Automorphisms of $X_{ns}(p)$ not preserving the cusps}

Let $p$ be a prime number and $l$ be another prime number different from $p$. Let $T_l$ be the $l$-th Hecke operator for the modular curve $X_{ns}(p)$, defined as a modular correspondence (see \cite[Section 2]{ShimuraCorr}). Let $E$ be an elliptic curve over $\mathbb C$ and let $\varphi$ be an isomorphism from the $p$-torsion $E[p]$ of $E$ to $\mathbb Z/p\mathbb Z\times \mathbb Z/p\mathbb Z$. Then the pair $(E,\varphi)$ gives a point on $X_{ns}(p)$. Let $A$ be a subgroup of $E$ of order $l$. Since $l\neq p$, taking the quotient by $A$ on $E$ induces an isomorphism from the $p$-torsion of $E$ to the $p$-torsion of $E/A$ so that $\varphi$ induces an isomorphism $\varphi_A$ from the $p$-torsion of $E/A$ to $\mathbb Z/p\mathbb Z\times \mathbb Z/p\mathbb Z$. The operator $T_l$ acts on a non-cuspidal point of $X_{ns}(p)$ in the following way
$$T_l(E,\varphi)=\sum_{\scriptsize\begin{matrix}A\subset E(\mathbb C)\\\#A=l\end{matrix}}(E/A,\varphi_A)$$
and it induces an element in the endomorphism algebra of $J_{ns}(p)$.

For every pair $(E,\varphi)$ let $\widetilde E$ be the reduction of $E$ modulo $l$ and $\widetilde \varphi:\widetilde E[p]\rightarrow\mathbb Z/p\mathbb Z\times\mathbb Z/p\mathbb Z$ be the map naturally induced by $\varphi$, since $l\neq p$. Then the pair $(\widetilde E,\widetilde\varphi)$ gives a point on $X_{ns}(p)$ over $\overline{\mathbb F_l}$. Let $\text{Frob}_l$ be the generator of the Galois group $\text{Gal}(\overline{\mathbb F_l}/\mathbb F_l)$ and let $[l^{-1}]$ be the inverse of the multiplication by $l$ map on $E[p]$. We have
$$\sum_{\scriptsize\begin{matrix}A\subset \widetilde E(\mathbb \overline{\mathbb F_l})\\\#A=l\end{matrix}}(\widetilde E/A,\widetilde \varphi_A)=(\widetilde E^{\textnormal{Frob}_l},\widetilde\varphi\circ{\textnormal{Frob}_l}^{-1})+l\cdot (\widetilde E^{{\textnormal{Frob}_l}^{-1}},\widetilde \varphi\circ[l^{-1}]\circ\textnormal{Frob}_l)$$
as divisors of $X_{ns}(p)$ reduced modulo $l$ (see, for example, the proofs in \cite[Section 8.7]{DS} regarding the modular curve $X_1(N)$). This implies the Eichler-Shimura relation:
$$T_l=\text{Frob}_l+l\cdot\text{Frob}_l^{-1}$$
where both terms of the equation are homomorphisms on the divisor group of $X_{ns}(p)$ reduced modulo $l$. Note that the Eichler-Shimura relation acquires this form in this case because a non-split Cartan subgroup of $\text{GL}_2(\mathbb F_p)$ always contains the scalar matrix associated to $[l^{-1}]$, hence the pairs $(\widetilde E^{{\textnormal{Frob}_l}^{-1}},\widetilde \varphi\circ[l^{-1}]\circ\textnormal{Frob}_l)$ and $(\widetilde E^{{\textnormal{Frob}_l}^{-1}},\widetilde \varphi\circ\textnormal{Frob}_l)$ give the same point on $X_{ns}(p)$.

Since $T_l$ induces an endomorphism of $J_{ns}(p)$, we will also write, with abuse of notation, the relation $T_l=\text{Frob}_l+l\cdot\text{Frob}_l^{-1}$ as endomorphisms of $J_{ns}(p)$ reduced modulo $l$.

\begin{lem}\label{uTl=Tlu}
Let $u$ be an automorphism of $X_{ns}(p)$ defined over $K(p)$. Let $\sigma_l\in\text{Gal}(\overline{\mathbb Q}/\mathbb Q)$ be a Frobenius element at $l$. Then if $l\neq p$ we have
$$u^{\sigma_l}T_l=T_lu$$
as endomorphisms of $J_{ns}(p)$. i.e.
$$
T_lu=\left\{\begin{matrix}uT_l&\textnormal{if $l$ is a square modulo $p$}\\
\overline u T_l&\textnormal{if $l$ is not a square modulo $p$} \end{matrix}\right.
$$
where $\overline u$ is the $\textnormal{Gal}(K(p)/\mathbb Q)$-conjugate of $u$.
\end{lem}
\begin{proof}
(Following \cite[p. 64, Lemma 2.6]{KM}) Since $J_{ns}(p)$ has good reduction outside $p$, the reduction modulo $l$ map is injective on the endomorphism ring of $J_{ns}(p)$. Hence, the Lemma follows from the Eichler-Shimura relations and the fact that $K(p)$ is a quadratic field.
\end{proof}

Let $\zeta_p$ be a primitive $p$-th root of unity. The modular curve $X_{ns}(p)$ has exactly $p-1$ cusps, all defined over $\mathbb Q(\zeta_p)$ and all conjugate under the action of $\text{Gal}(\mathbb Q(\zeta_p)/\mathbb Q)$ (see \cite[p. 194-195]{Ser89}).
\begin{prop}
Let $\sigma_l\in\text{Gal}(\overline{\mathbb Q}/\mathbb Q)$ be a Frobenius element at $l$. Then we have
$$T_lC=C^{\sigma_l}+l\cdot C^{{\sigma_l}^{-1}}$$
as divisors on $X_{ns}(p)$, for every cusp $C$ of $X_{ns}(p)$.
\end{prop}
\begin{proof}
The equality must hold modulo any prime ideal $\mathfrak l$ of $\mathbb Q(\zeta_p)$ over $l$ because of the Eichler-Shimura relations. Moreover, since $l\neq p$, the modular curve $X_{ns}(p)$ and its reduction modulo $l$ have the same number of cusps. Therefore, reduction modulo $\mathfrak l$ is injective on the set of cusps of $X_{ns}(p)$, and the equality holds also over $\mathbb Q(\zeta_p)$.
\end{proof}

\begin{lem}\label{Dl0}
Let $p\ge 11$ and let $u$ be an automorphism of $X_{ns}(p)$ defined over $K(p)$. Let $\sigma_l\in\textnormal{Gal}(\overline{\mathbb Q}/\mathbb Q)$ be a Frobenius element at $l$ and let $C$ be a cusp of $X_{ns}(p)$. If for every cusp $C'$ of $X_{ns}(p)$ the divisor
\begin{align*}
D_{l}\mathop{=}^\textnormal{def}&(u^{\sigma_l}T_l-T_lu)(C-C')\\
=&u^{\sigma_l}T_lC+T_luC'-T_luC-u^{\sigma_l}T_lC'
\end{align*}
is the zero divisor, then $uC$ is a cusp. i.e.

If $u^{\sigma_l}T_l-T_lu$ is the zero operator on the group of divisors of $X_{ns}(p)$ of degree zero and supported in the cusps, then $u$ preserves the cusps.
\end{lem}
\begin{proof}
By the previous Proposition, for every cusp $C$, we have $T_l C=C^{\sigma_l}+l\cdot C^{{\sigma_l}^{-1}}$. Thus, we can choose a cusp $C'$ such that the supports of $T_lC$ and $T_lC'$ are completely disjoint. For $D_l$ to be the zero divisor, we must have
$$u^{\sigma_l}C^{\sigma_l}+l\cdot u^{\sigma_l}C^{{\sigma_l}^{-1}}+T_luC'=T_luC+u^{\sigma_l}C'^{\sigma_l}+l\cdot u^{\sigma_l}C'^{{\sigma_l}^{-1}}\text{.}$$
Since we chose $C'$ such that the supports of $T_lC$ and $T_lC'$ are completely disjoint, we have
\begin{equation}\label{Tlu}
T_luC=u^{\sigma_l}C^{\sigma_l}+l\cdot u^{\sigma_l}C^{{\sigma_l}^{-1}}\text{.}
\end{equation}
If $uC$ is not a cusp, it corresponds to an elliptic curve $E$ defined over $\mathbb Q(uC)\subset\mathbb Q(\zeta_p)$.

Let $A$ and $B$ be two different cyclic groups of order $l$ in $E$. Then the natural map $E/A\rightarrow E\rightarrow E/B$ has a cyclic kernel of order $l^2$. Therefore, recalling the interpretation of the action of $T_l$ on non-cuspidal points, we have that equation (\ref{Tlu}) can hold only if $E$ admits an endomorphism of degree $l^2$ with cyclic kernel. An elliptic curve without complex multiplication cannot have such an endomorphism, hence we obtain the Lemma as a consequence of the following Proposition.
\end{proof}
\begin{prop}
Let $p\ge 11$ and let $E$ be an elliptic curve over $\mathbb C$ with complex multiplication by an imaginary quadratic field $K$. No point of $X_{ns}(p)$ associated to $E$ is defined over $\mathbb Q(\zeta_p)$.
\end{prop}
\begin{proof}
Let $P$ be a point of $X_{ns}(p)$ associated to $E$ and defined over $\mathbb Q(\zeta_p)$. There are three cases: $p$ is inert, $p$ splits or $p$ ramifies in $K$.

If $p$ is inert in $K$, we consider the point $Q$ of $X_{ns}^+(p)$ given by $\{P,wP\}$. This point $Q$ is defined over $L\subset\mathbb Q(\zeta_p)$ and lifts to points of $X_{ns}(p)$ defined over $LK$ as explained in \cite[p. 194-195]{Ser89}. But this is impossible since $P,wP$ are defined over $\mathbb Q(\zeta_p)$ where $p$ ramifies completely.

If $p$ splits or ramifies in $K$, this means that the image of the $\text{Gal}(\overline{\mathbb Q}/K(\zeta_p))$-representation modulo $p$ attached to $E$ is contained in respectively a split Cartan or a Borel subgroup of $\text{GL}_2(\mathbb F_p)$ (see loc. cit.). But since $P$ is defined over $\mathbb Q(\zeta_p)$, the image of such representation is also contained in a non-split Cartan subgroup. Therefore, in both cases, it is contained in the subgroup of scalar matrices. Then, the properties of the Weil pairing imply that $\text{Gal}(\overline{\mathbb Q}/K(\zeta_p))$ acts on the $p$-torsion of $E$ as a subgroup of $\{\pm1\}$, which means that the $x$-coordinates of the $p$-torsion points of $E$ are in $K(\zeta_p)$. The theory of complex multiplication tells us that the ray class field $K_{(p)}$ modulo $(p)$ of $K$ is generated over $K$ by the $j$-invariant of $E$ and by some rational function of the $x$-coordinates of the $p$-torsion points of $E$ (\cite[p. 135, Theorem 5.6]{SilvermanAT}). Thus we obtain $K_{(p)}\subset K(\zeta_p)$. It is always true that $K(\zeta_p)\subset K_{(p)}$, so the opposite inclusion holds if and only if $[K_{(p)}:K]\le [K(\zeta_p):K]$. We have
$$[K_{(p)}:K]=h_K\frac{\#(O_K/pO_K)^*}{w_K}$$
where $h_K$ is the class number of $K$, $O_K$ is the ring of integers of $K$ and $w_K$ is the number of roots of unity in $K$ taken modulo $(p)$. The group $(O_K/pO_K)^*$ is isomorphic to $\mathbb F_{p^2}^*$, $\mathbb F_p^*\times\mathbb F_p^*$ or $\mathbb F_p^*\times \langle a\rangle$ with $a$ an element of order $p$. Thus to have the inequality $[K_{(p)}:K]\le [K(\zeta_p):K]$ we must have
$$\frac{(p-1)^2}{w_K}\le (p-1)$$
and hence $p\le w_K+1$. Since the roots of unity in an imaginary quadratic field are at most $6$, we get $p\le 7$.
\end{proof}

\begin{cor}\label{wfnotf}
Let $p\ge 11$ and let $u$ be an automorphism of $X_{ns}(p)$ defined over $K(p)$. If $u$ does not preserve the cusps, then there exists a nonconstant morphism $f$ from $X_{ns}(p)$ to $\mathbb P^1$, defined over the field $\mathbb Q(\zeta_p)$ and with degree less or equal to $8$, such that $f\circ w\neq f$.
\end{cor}
\begin{proof}
Since $u$ is defined over $K(p)$ and does not preserve the cusps, we can apply Lemma \ref{Dl0} to obtain, for every prime $l\neq p$, a non-zero divisor $D_l$ on $X_{ns}(p)$, of degree zero, defined over $\mathbb Q(\zeta_p)$, which is the difference of two effective divisors of degree at most $2(l+1)$. Moreover, we have $D_{l}=(u^{\sigma_l}T_l-T_lu)(C-C')$ for two cusps $C$, $C'$ of $X_{ns}(p)$. Hence, by Lemma \ref{uTl=Tlu} we deduce that $D_l$ must be zero in $J_{ns}(p)$ when $p\ge 11$, which is to say that $D_l$ is a principal divisor. This gives us the existence of a nonconstant morphism $f$ from $X_{ns}(p)$ to $\mathbb P^1$, defined over $\mathbb Q(\zeta_p)$ and with degree at most $2(l+1)$.

Now we apply $w$ to $D_l$. If $D_l$ were invariant under the action of $w$, we would have
\begin{equation}\label{wDl=Dl}
wT_luC+wu^{\sigma_l}C'^{\sigma_l}+l\cdot wu^{\sigma_l}C'^{{\sigma_l}^{-1}}=T_luC+u^{\sigma_l}C'^{\sigma_l}+l\cdot u^{\sigma_l}C'^{{\sigma_l}^{-1}}\text{,}
\end{equation}
where we used the computation in the proof of Lemma \ref{Dl0}. In that proof, we also showed that the divisor $T_luC$ is the sum of $l+1$ different points. Hence if $l\ge 3$, to satisfy equation \ref{wDl=Dl}, we must have $wu^{\sigma_l}C'^{{\sigma_l}^{-1}}=u^{\sigma_l}C'^{{\sigma_l}^{-1}}$. But this is impossible by the previous Proposition, because the fixed points of $w$ are all associated with the elliptic curve with $j$-invariant equal to $1728$ (see \cite[Proposition 7.10]{BaranClass}), which is an elliptic curve with complex multiplication. Thus, for $l\ge 3$, we have $wD_l\neq D_l$ which implies $f\circ w\neq f$. The degree of $f$ is $2(l+1)$, so by choosing $l=3$ we get the Corollary.
\end{proof}

\section{Proof of main results}

We first introduce a general Lemma.
\begin{lem}
Let $X$ be a smooth projective curve over an algebraically closed field, and let $w$ be a nontrivial automorphism of $X$ having $r$ fixed points. If there exists a morphism $f$ from $X$ to $\mathbb P^1$ of degree $k$ such that $f\circ w\neq f$, then $r\le 2k$.
\end{lem}
\begin{proof}
(\cite[Lemma 3.5]{BakerHas}) Let $g=f\circ w-f$. If $g$ is constant, then the fixed points of $w$ must be poles of $f$, implying $r\leq k$. If $g$ is not constant then the degree of $g$ is at most $2k-h$, where $h$ is the number of poles that are also points fixed by $w$. On the other hand, the other $r-h$ points fixed by $w$ are zeros of $g$, hence there are at most $\text{deg }g$ of them. So we have $r-h\le 2k-h$ which implies $r\le 2k$.
\end{proof}

Let $p$ be an odd prime number.
\begin{thm}\label{aut}
If $p\geq 37$, all the automorphisms of $X_{ns}(p)$ preserve the cusps. If $p\equiv 1\textnormal{ mod }12$ and $p\neq 13$, the automorphism group of $X_{ns}(p)$ is generated by the modular involution $w$.
\end{thm}
\begin{proof}
Let $u$ be an automorphism of $X_{ns}(p)$ not preserving the cusps. By Corollary \ref{Kp} we have that $u$ is defined over $K(p)$. Then Corollary \ref{wfnotf} tells us that there exists a nonconstant morphism from $X_{ns}(p)$ to $\mathbb P^1$, defined over $\mathbb Q(\zeta_p)$, with degree less or equal to $8$ and such that $f\circ w\neq f$. Then by the previous Lemma we have
$$\#\{\textnormal{fixed points of }w\}\le 16\text{.}$$
Recall that the number of fixed points of $w$ is $\frac{p-1}2$ if $p\equiv 1\text{ mod }4$ and it is $\frac{p+1}2$ if $p\equiv 3\text{ mod }4$ (see \cite[Proposition 7.10]{BaranClass}). This implies $p\le 31$ and proves the first part of the Theorem. The second part is now a consequence of Corollary \ref{ellcuspModular}.
\end{proof}

\bigskip
Now we prove that Serre's uniformity conjecture, implies the absence of rational exceptional automorphisms of $X_{ns}(p)$ for almost all $p$.

\begin{thm}\label{unifAut}
Let $p\geq 37$ and assume $X_{ns}(p)$ has a CM rational point (i.e. $p$ is inert in at least one of the imaginary quadratic field of class number one). If there exists an exceptional automorphism of $X_{ns}(p)$ defined over $\mathbb Q$, then $X_{ns}^+(p)$ has a non-CM rational point.
\end{thm}
\begin{proof}
Let $u$ be an automorphism of $X_{ns}(p)$ defined over $\mathbb Q$. By Corollary \ref{Qcomw} we have that $u$ commutes with $w$ and it induces an automorphism of $X_{ns}^+(p)$ defined over $\mathbb Q$. Then it is enough to prove that $u$ preserves the rational CM points of $X_{ns}^+(p)$ only if $u$ is modular. Let $\pi:X_{ns}(p)\rightarrow X_{ns}^+(p)$ be the degree-2 modular morphism given by the inclusion of a non-split Cartan subgroup in its normalizer, and let $Q$ be a rational CM point of $X_{ns}^+(p)$. Note that the two points in $\pi^{-1}(Q)$ are defined over the CM field of the elliptic curve associated to $Q$ (see \cite[p. 194-195]{Ser89}). Hence $uQ$ cannot be a rational CM point different from $Q$ because the two points in $\pi^{-1}(uQ)$ are defined over the same field as the points in $\pi^{-1}(Q)$. Now we observe that the elliptic points of $X_{ns}(p)$ are the inverse images by $\pi$ of the rational CM points associated to the elliptic curves with $j$-invariant equal to 0 or 1728 (see \cite[Proposition 7.10]{BaranClass}). Thus, if $u$ preserves the rational CM points of $X_{ns}^+(p)$, then it preserves the elliptic points of $X_{ns}(p)$. Furthermore, $u$ also preserves the cusps by Theorem \ref{aut}, hence it is modular by Corollary \ref{ellcuspModular}.
\end{proof}

\bibliographystyle{amsalpha}
\bibliography{Automorfismi_arXiv}{}

\def\cprime{$'$}
\providecommand{\bysame}{\leavevmode\hbox to3em{\hrulefill}\thinspace}
\providecommand{\MR}{\relax\ifhmode\unskip\space\fi MR }
\providecommand{\MRhref}[2]{%
  \href{http://www.ams.org/mathscinet-getitem?mr=#1}{#2}
}
\providecommand{\href}[2]{#2}
\begin{thebibliography}{BGJGP05}

\bibitem[Bar10]{BaranClass}
B.~Baran, \emph{Normalizers of non-split {C}artan subgroups, modular curves,
  and the class number one problem}, J. Number Theory \textbf{130} (2010),
  no.~12, 2753--2772. \MR{2684496 (2011i:11083)}

\bibitem[BGJGP05]{BJG}
Matthew~H. Baker, Enrique Gonz{\'a}lez-Jim{\'e}nez, Josep Gonz{\'a}lez, and
  Bjorn Poonen, \emph{Finiteness results for modular curves of genus at least
  2}, Amer. J. Math. \textbf{127} (2005), no.~6, 1325--1387. \MR{2183527
  (2006i:11065)}

\bibitem[BH03]{BakerHas}
Matthew Baker and Yuji Hasegawa, \emph{Automorphisms of {$X_0^*(p)$}}, J.
  Number Theory \textbf{100} (2003), no.~1, 72--87. \MR{1971247 (2004c:11100)}

\bibitem[Che98]{Chen}
Imin Chen, \emph{The {J}acobians of non-split {C}artan modular curves}, Proc.
  London Math. Soc. (3) \textbf{77} (1998), no.~1, 1--38. \MR{1625491
  (99m:11068)}

\bibitem[Con96]{Conway}
J.~H. Conway, \emph{Understanding groups like {$\Gamma_0(N)$}}, Groups,
  difference sets, and the {M}onster ({C}olumbus, {OH}, 1993), Ohio State Univ.
  Math. Res. Inst. Publ., vol.~4, de Gruyter, Berlin, 1996, pp.~327--343.
  \MR{1400424 (98b:11041)}

\bibitem[Dav80]{Davenport}
Harold Davenport, \emph{Multiplicative number theory}, second ed., Graduate
  Texts in Mathematics, vol.~74, Springer-Verlag, New York-Berlin, 1980,
  Revised by Hugh L. Montgomery. \MR{606931 (82m:10001)}

\bibitem[DFGS14]{Cat}
Valerio Dose, Julio Fern{\'a}ndez, Josep Gonz{\'a}lez, and Ren{\'e} Schoof,
  \emph{The automorphism group of the non-split {C}artan modular curve of level
  11}, J. Algebra \textbf{417} (2014), 95--102. \MR{3244639}

\bibitem[DR73]{DelRap}
P.~Deligne and M.~Rapoport, \emph{Les sch\'emas de modules de courbes
  elliptiques}, Modular functions of one variable, {II} ({P}roc. {I}nternat.
  {S}ummer {S}chool, {U}niv. {A}ntwerp, {A}ntwerp, 1972), Springer, Berlin,
  1973, pp.~143--316. Lecture Notes in Math., Vol. 349. \MR{0337993 (49
  \#2762)}

\bibitem[DS05]{DS}
Fred Diamond and Jerry Shurman, \emph{A first course in modular forms},
  Graduate Texts in Mathematics, vol. 228, Springer-Verlag, New York, 2005.
  \MR{2112196 (2006f:11045)}

\bibitem[dSE00]{Edix}
Bart de~Smit and Bas Edixhoven, \emph{Sur un r\'esultat d'{I}min {C}hen}, Math.
  Res. Lett. \textbf{7} (2000), no.~2-3, 147--153. \MR{1764312 (2001j:11043)}

\bibitem[Elk90]{Elkies}
Noam~D. Elkies, \emph{The automorphism group of the modular curve {$X_0(63)$}},
  Compositio Math. \textbf{74} (1990), no.~2, 203--208. \MR{1047740
  (91e:11064)}

\bibitem[Har11]{Harrison108}
Michael~Corin Harrison, \emph{A {N}ew {A}utomorphism {O}f {$X_0(108)$}},
  arXiv:1108.5595 (2011).

\bibitem[KM88]{KM}
M.~A. Kenku and F.~Momose, \emph{Automorphism groups of the modular curves
  {$X_0(N)$}}, Compositio Math. \textbf{65} (1988), no.~1, 51--80. \MR{930147
  (88m:14015)}

\bibitem[Lan02]{LungLang2002}
Mong~Lung Lang, \emph{Normalisers of subgroups of the modular group}, J.
  Algebra \textbf{248} (2002), no.~1, 202--218. \MR{1879013 (2002m:20075)}

\bibitem[Ogg74]{OggHyp}
A.~P. Ogg, \emph{Hyperelliptic modular curves}, Bull. Soc. Math. France
  \textbf{102} (1974), 449--462. \MR{0364259 (51 \#514)}

\bibitem[Ogg77]{OggDe}
\bysame, \emph{\"{U}ber die {A}utomorphismengruppe von {$X_{0}(N)$}}, Math.
  Ann. \textbf{228} (1977), no.~3, 279--292. \MR{0562500 (58 \#27775)}

\bibitem[Rib75]{RibetEnd}
Kenneth~A. Ribet, \emph{Endomorphisms of semi-stable abelian varieties over
  number fields}, Ann. Math. (2) \textbf{101} (1975), 555--562. \MR{0371903 (51
  \#8120)}

\bibitem[Ser89]{Ser89}
J.-P. Serre, \emph{Lectures on the {M}ordell-{W}eil {T}heorem}, Aspects
  {M}ath., vol. E15, Springer Vieweg, Braunschweig/Wiesbaden, 1989.

\bibitem[Shi58]{ShimuraCorr}
Goro Shimura, \emph{Correspondances modulaires et les fonctions {$\zeta $} de
  courbes alg\'ebriques}, J. Math. Soc. Japan \textbf{10} (1958), 1--28.
  \MR{0095173 (20 \#1679)}

\bibitem[Shi71]{ShimuraCM}
\bysame, \emph{On elliptic curves with complex multiplication as factors of the
  {J}acobians of modular function fields}, Nagoya Math. J. \textbf{43} (1971),
  199--208. \MR{0296050 (45 \#5111)}

\bibitem[Shi72]{ShimuraClassFields}
\bysame, \emph{Class fields over real quadratic fields and {H}ecke operators},
  Ann. of Math. (2) \textbf{95} (1972), 130--190. \MR{0314801 (47 \#3351)}

\bibitem[Sil94]{SilvermanAT}
Joseph~H. Silverman, \emph{Advanced topics in the arithmetic of elliptic
  curves}, Graduate Texts in Mathematics, vol. 151, Springer-Verlag, New York,
  1994. \MR{1312368 (96b:11074)}

\bibitem[Sil09]{Silverman}
\bysame, \emph{The arithmetic of elliptic curves}, second ed., Graduate Texts
  in Mathematics, vol. 106, Springer, Dordrecht, 2009. \MR{2514094
  (2010i:11005)}

\bibitem[Ste12]{SteinTab}
William Stein, \emph{The {M}odular {F}orms {D}atabase},
  \newline\url{http://wstein.org/Tables}.

\end{thebibliography}

\end{document}